\documentclass{svjour3}                     
\smartqed  
\usepackage{graphicx}
\spnewtheorem{algorithm}{Algorithm}{\bf}{\rm}
\begin{document}

\title{Discovering a junction tree behind a Markov network by a greedy algorithm
}

\titlerunning{Hypergraphs and dependence structures}        

\author{Tam\'{a}s Sz\'{a}ntai         \and
        Edith Kov\'{a}cs 
}

\authorrunning{T. Sz\'{a}ntai and E. Kov\'{a}cs} 

\institute{Institute of Mathematics, Budapest University of Technology and Economics,
            \at M\H{u}egyetem rkp. 3, H-1111 Budapest, Hungary \\
              Tel.: +36-1-463-1298\\
              Fax: +36-1-463-1291\\
              \email{szantai@math.bme.hu}           
           \and
           Department of Mathematics, \'{A}VF College of Management of Budapest \at
              Vill\'{a}nyi \'{u}t 11-13, H-1114 Budapest, Hungary
}

\date{Received: date / Accepted: date}

\maketitle

\begin{abstract}
In our paper \cite{SzaKov08} we introduced a special kind of $k$-width junction tree,
called $k$-th order $t$-cherry junction tree in order to approximate a joint
probability distribution. The approximation is the best if the
Kullback-Leibler divergence between the true joint probability distribution
and the approximating one is minimal. Finding the best approximating $k$-width
junction tree is NP-complete if $k>2$ (see in \cite{KarSre01}). In \cite{SzaKov10} we
also proved that the best approximating $k$-width junction tree can be embedded
into a $k$-th order $t$-cherry junction tree. We introduce a greedy
algorithm resulting very good approximations in reasonable computing time.

In this paper we prove that if the Markov network underlying fullfills some
requirements then our greedy algorithm is able to find the true probability
distribution or its best approximation in the family of the $k$-th order 
$t$-cherry tree probability distributions. 
Our algorithm uses just the $k$-th order marginal probability
distributions as input.

We compare the results of the greedy algorithm proposed in this paper
with the greedy algorithm proposed by Malvestuto \cite{Mal91}.
\keywords{Greedy algorithm \and Conditional independence \and Markov network 
\and Triangulated graph \and Graphical models \and $t$-cherry junction tree
\and Contingency table}
\subclass{90C35 \and 90C59 \and 62B10 \and 94A17 \and 62H05 \and 62H30}
\end{abstract}

\section{Introduction}
\label{intro}
The problem of approximating multivariate probability distributions is a
central task of many fields. Unfortunately in most of the cases we know
nothing about the theoretical probability distribution. It is useful to
exploit the dependence structure between the random variables involved. The
problem is: what should we do when correlation matrices can not be used.

Starting from a discrete probability distribution, for example from a sample
data, it is useful to discover some of the conditional independences
between the variables.

The Markov networks (Markov random fields) and Bayesian networks encode
these conditional independences. In our paper we focus on the Markov
networks. If the graph structure of the Markov network is known, many
procedures were developed for its inference, see \cite{Pea88} and \cite{CowDawLauSpi03}. There are
many cases where the graph structure of the Markov network is unknown. 
In \cite{SzaKov08} we proposed a method for discovering some of the
conditional independences between the random variables by fitting a special 
type of multivariate probability distribution called
$t$-cherry junction tree distribution to the 
sample data. The goodness of fit was quantified by
the Kullback-Leibler divergence (see \cite{Kul59}). This relates the problem to information
theory (\cite{CoTo91}). On the other side, the graph underlying the Markov network links
the problem to graph theory. For elements of graph theory see \cite{Berge73}.

In the second section we introduce some concepts used in graph theory and
probability theory that we need throughout the paper and present how these
can be linked to each other. For a good overview see \cite{LauSpi88}.

In the third part we introduce the Sz\'{a}ntai-Kov\'{a}cs's greedy algorithm
which starting from the $k$-th order marginal probability distributions gives
a $k$-th order $t$-cherry junction tree probability distribution as a result.
For the same task Malvestuto gives another algorithm in \cite{Mal91}. First we compare
these two algorithms from analytical point of view and then apply them on the
example problem presented in Malvestuto's paper \cite{Mal91}.

In the fourth part we introduce the so called puzzle algorithm for $k$-th order $t$-cherry trees. 
This results in a puzzle numbering of the verticies. Using this we give some theoretical
results related to our greedy algorithm.

The last part contains conclusions and some possible applications of our
greedy algorithm.

\section{Preliminaries}
\label{sec:2}
This part contains a summary of the concepts used throughout the paper. We first
present the acyclic hypergraphs and junction trees. We then present a short
reminder on Markov network. We finish this part with the multivariate joint
probability distribution associated to a junction tree.

Let $V=\left\{ 1,\ldots ,d\right\} $ be a set of vertices and $\Gamma $ a
set of subsets of $V$ called {\it set of hyperedges}. A \textit{hypergraph}
consists of a set $V$ of vertices and a set $\Gamma $ of hyperedges. We
denote a hyperedge by $C_i$, where $C_i$ is a subset of $V$. If two vertices
are in the same hyperedge they are connected, which means, the hyperedge of a
hyperhraph is a complete graph on the set of vertices contained in it.

A vertex is called \textit{simplicial} if it belongs to precisely one
hyperedge.

An ordering of the vertices is a \textit{perfect elimination ordering} if 
$ \forall i, 1\leq i\leq d$ the vertex $i$ is simplicial in the subhypergraph defined on
the vertices $\left\{ i,\ i+1,\ldots ,d\right\} .$

The \textit{acyclic}{\normalsize \ }\textit{hypergraph} is a special type of
hypergraph which fulfills the following requirements:

\begin{itemize}
\item  Neither of the edges of $\Gamma $ is a subset of another edge.

\item  There exists a numbering of edges for which the \textit{running
intersection property} is fullfiled: $\forall j\geq 2\quad \ \exists \ i<j:\
C_i\supset C_j\cap \left( C_1\cup \ldots \cup C_{j-1}\right) $. (Other
formulation is that for all hyperedges $C_i$ and $C_j$ with $i<j-1$,
$C_i \cap C_j \subset C_s \ \mbox{for all} \ s, i<s<j$.)
\end{itemize}

Let $S_j=C_j\cap \left( C_1\cup \ldots \cup C_{j-1}\right) $, for $j>1$ and $%
S_1=\phi $. Let $R_j=C_j\backslash S_j$. We say that $S_j$\textit{separates} 
$R_j$ from $\left( C_1\cup \ldots \cup C_{j-1}\right) \backslash S_j$, and
call $S_j$ separator set or shortly separator.

Now we link these concepts to the terminology of junction trees.

The junction tree is a special tree stucture which is equivalent to the
connected acyclic hypergraphs \cite{LauSpi88}. The nodes of the tree correspond
to the hyperedges of the connected acyclic hypergraph and are called clusters, the edges of the tree
correspond to the separator sets and called separators. The set of all
clusters is denoted by $\mathcal{C}$, the set of all separators is denoted by %
$\mathcal{S}$. The junction tree with the largest cluster containing $k$
variables is called \textit{k-width junction tree}.

An important relation between graphs and hypergraphs is given in \cite{LauSpi88}: A
hypergraph is acyclic if and only if it can be considered to be the set of
cliques of a triangulated graph (a graph is triangulated if every cycle of
legth greater than 4 has a chord).

\begin{theorem}
\label{theo:1}
(Fulkerson and Gross, \cite{FulGro65}): A graph is an acyclic hypergraph (triangulated graph
or junction tree) if and only if has an perfect elimination ordering.
\end{theorem}

\begin{algorithm}
\label{alg:1}
(Graham, \cite{Gra79}) A \textit{Graham reduction} of a hypergraph $H=(V,\Gamma )$ is defined by
applying the following two operations to \textit{H} until they can be
applied no more.

\begin{itemize}
\item  Node removal: If a node appears in only one hyperedge, delete it from
V and from the edge.

\item  Hyperedge removal: In the the transformed hyperedge set, delete a
hyperedge if it is subset of another hyperedge.
\end{itemize}
\end{algorithm}

In \cite{BeeFagMaiYan83} is shown that a hypergraph reducies to nothing by this process if
and only if the hypergraph is acyclic.

In the Figure \ref{fig:1} one can see a) a triangulated graph, b) the
corresponding acyclic hypergraph and c) the corresponding junction tree.

\begin{figure}
  \includegraphics[bb=80 590 420 770,width=10cm]{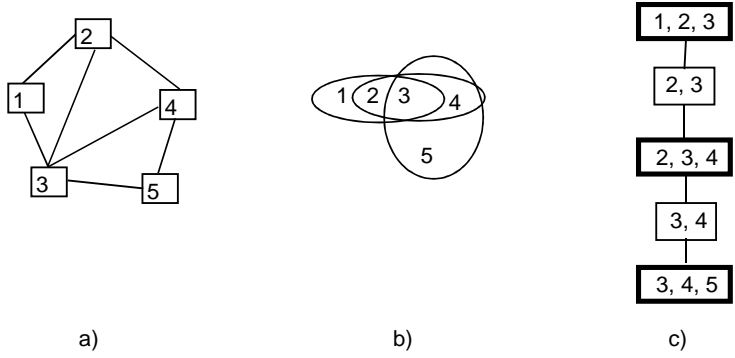}
\caption{a) Triangulated graph, b) The corresponding acyclic hypergraph, c) The corresponding junction tree}
\label{fig:1}       
\end{figure}

We consider the random vector $\mathbf{X}=\left( X_1,\ldots ,X_d\right) ^T$, with the
set of indicies $V=\left\{ 1,\ldots ,d\right\} $. Roughly speaking a Markov
network encodes the conditional independences between the random variables.
The graph structure associated to a Markov network consists in the set of
nodes V, and the set of edges $E=\left\{ \left( i,j\right) |i,j\in V\right\} 
$. We say the graph structure associated to the Markov network has

\begin{itemize}
\item  the \textit{pairwise Markov} (PM) property if $\forall i,j\in V$, $i$
not connected to $j$ implies that $X_i$ and $X_j$ are conditionally
independent given all the other random variables;

\item  the \textit{local Markov (LM)} property if $\forall i\in V,$ and
$ Ne\left( i\right) $ the neighbourhood of node $i$
in the graph (the nodes connected with $i$) then $X_i$ is conditionally independent from
all $X_j$, $j\notin Ne\left( i\right) $, given $X_k,k\in Ne\left( i\right) $;

\item  the global Markov (GM) property states that if in the graph 
$\forall A,B,C\subset V$ and C separates A and B in terms of graph then $\mathbf{X}_A$%
and $\mathbf{X}_B$ are conditionally independent given $\mathbf{X}_C$, which means in terms of
probabilities that 
\[P\left( \mathbf{X}_{A\cup B\cup C}\right) =\frac{P\left( \mathbf{X}_{A\cup
C}\right) P\left( \mathbf{X}_{A\cup C}\right) }{P\left( \mathbf{X}_C\right) };
\]

\item  the \textit{factorization (F)} property states that if $\mathcal{C}$
denotes the set of cliques of the graph (maximum complete graphs) then there
exist positive functions $\psi _C\left( \mathbf{X}_C\right)$ that
\[
P\left( \mathbf{X}_V\right) =\prod\limits_{C\in \mathcal{C}}\psi _C\left( \mathbf{X}_C\right).
\]
\end{itemize}

The following implication is well known \cite{Ham71}: $F\Rightarrow GM\Rightarrow
LM\Rightarrow PM$. The Hammersley-Clifford theorem states that under
assumption of positivity $PM\Longrightarrow F$. However positivity is a very
strong condition. ''The positivity condition is mathematically convenient; But it hardly
seems necessary'' \cite{Ham71}. In this paper we focus on Markov network
characterized by the global Markov property.

The concept of \textit{junction tree probability distribution} is related to
the junction tree graph and to the global Markov property of the graph. A
junction tree probability distribution is defined as a product and division of
marginal probability distributions as follows:

\[
P\left( \mathbf{X}\right) =\frac{\prod\limits_{C\in \mathcal{C}}P\left(
\mathbf{X}_C\right) }{\prod\limits_{S\in \mathcal{S}}\left[ P\left( \mathbf{X}_S\right)
\right] ^{\nu _S-1}},
\]
where $\mathcal{C}$ is the set of clusters of the
junction tree, $\mathcal S$ is the set of separators, $\nu_S$ is the number of those clusters
which contain the separator S. We emphasize here that the equalities written as
$P(\mathbf{X})=f(P(\mathbf{X}_K), K\in \mathcal{C})$, where $f: \Omega_{\mathbf{X}}\rightarrow R$
hold for any possible realization of $\mathbf{X}$.

\begin{example}
\label{ex:1}
The probability distribution corresponding to Figure \ref{fig:1} is:
\[
\begin{array}{rcl}
P\left( \mathbf{X}\right) & = & \frac{P\left( \mathbf{X}_{\{1,2,3\}}\right) P\left(
\mathbf{X}_{\{2,3,4\}}\right) P\left(\mathbf{X}_{\{3,4,5\}}\right) }{P\left(\mathbf{X}_{\{2,3\}}\right)
P\left( \mathbf{X}_{\{3,4\}}\right) }\vspace{3mm} \\
 & = & \frac{P\left( X_1,X_2,X_3\right) P\left(
X_2,X_3,X_4\right) P\left( X_3,X_4,X_5\right) }{P\left( X_2,X_3\right)
P\left( X_3,X_4\right) }.
\end{array}
\]
\end{example}

In our paper \cite{SzaKov08} we introduced a special kind of $k$-width junction tree,
called $k$-th order $t$-cherry junction tree in order to approximate a joint
probability distribution. The $k$-th order $t$-cherry junction tree
probability distribution is associates to the $k$-th order $t$-cherry tree,
introduced in \cite{BuPre01}, \cite{BuSza02}.

\begin{definition}
\label{def:def1} 
The recursive construction of the \textit{k-th order
t-cherry tree}:

\begin{itemize}
\item  (i) The complete graph of $(k-1)$ nodes from $V$ represent the
smallest $k$-th order \textit{t}-cherry tree;

\item  (ii) By connecting a new vertex $i_k\in V$, with all $\left\{
i_1,\ldots ,i_{k-1}\right\} $ vertices of a $\left( k-1\right) $-
dimensional complete subgraph of the existing $k$-th order $t$-cherry tree, we
obtain a new $k$-th order $t$-cherry tree. $\left\{ \left\{ i_k\right\} \left\{
i_1,\ldots ,i_{k-1}\right\} \right\} $ is called \textit{k}-th order
hypercherry.

\item   (iii) A \textit{k}-th order $t$-cherry tree can be obtained from (i)
by successive application of (ii).
\end{itemize}

The \textit{k}-th order $t$-cherry tree is a special triangulated graph
therefore a junction tree structure is associated to it.
\end{definition}

\begin{definition}
\label{def:def2}
(\cite{SzaKov08})
The \textit{k-th order $t$-cherry junction tree} 
is defined in the following way:

\begin{itemize}
\item  By using Definition \ref{def:def1} we construct a \textit{k}-th order \textit{t}%
-cherry tree over $V$.

\item  To each hypercherry $\left\{ \left\{ i_k\right\} \left\{ i_1,\ldots
,i_{k-1}\right\} \right\} $ is assigned a cluster $\left\{ i_1,\ldots
,i_{k-1},i_k\right\}$ which a node of the junction tree and a separator $\left\{
i_1,\ldots ,i_{k-1}\right\} $ which is an edge of the junction tree.
\end{itemize}

We denote by $\mathcal{C}_{\mbox{ch}}$, and $\mathcal{S}_{\mbox{ch}}$, the
set of clusters and separators of the $t$-cherry junction tree.
\end{definition}

\begin{definition}
\label{def:def3}
(\cite{SzaKov08})
If the indices of the random vector $\mathbf{X}^T=\left( X_1,\ldots
,X_d\right)$ are assigned to a $t$-cherry junction tree structure then there
exists a probability distribution called 
\textit{\ t-cherry junction tree probability distribution} given by:

\[
P_{\mbox{t-ch}}(\mathbf{X})=\frac{\prod\limits_{C\in \mathcal{C}_{\mbox{ch}}}P\left( \mathbf{X}_C\right) }{%
\prod\limits_{S\in \mathcal{S}_{\mbox{ch}}}\left( P\left( \mathbf{X}_S\right) \right) ^{\nu _s-1}}%
.
\]
\end{definition}

\begin{remark}
\label{rem:rem1}
The marginal probability distributions involved in the above
formula are marginal probability distributions of $P\left( \mathbf{X}\right) $.
\end{remark}

Example \ref{ex:1} shows a 3-rd order $t$-cherry junction tree probability distribution.

In the following instead of probability distribution associated to a
junction tree we will use shortly junction tree pd and similarly instead of $k$-th order $t$-cherry
tree junction tree distribution we will use shortly $k$-th order $t$-cherry pd. Recently we found 
a paper \cite{Mal91} where Malvestuto introduced  
the same junction tree pd structure in a different way and named it \textit{elementary model
of rank k}.

The graph underlying the Markov network is usually unknown, the task of the
following section is to give a greedy algorithm, for finding a junction tree
starting from the $k$-th order marginal distributions, which are supposed to
be known.

\section{Sz\'{a}ntai-Kov\'{a}cs's greedy algorithm for finding an
approximating junction tree probability distribution}
\label{sec:3}

The problem is finding a $k$-width junction tree pd which gives the best
approximation for a discrete probability distribution $P(\mathbf{X})$. The
goodness of the approximation is quantified by the Kullback-Leibler
divergence, which have to be minimized:

\[
KL\left( P\left( \mathbf{X}\right) ,P_a\left( \mathbf{X}\right) \right) =%
\sum\limits_{\mathbf{x}}P\left( \mathbf{X}\right) \log _2\frac{P\left(
\mathbf{X}\right) }{P_a\left( \mathbf{X}\right) }\rightarrow \min.
\]

This minimization problem for $k>2$ can be solved in exact way
only by exhaustive search \cite{Pea88}. For k=2 the problem can be solved using
Kruskall's algorithm, as was first proposed by Chow and Liu \cite{Chow68}.

Malvestuto \cite{Mal91} and Sz\'{a}ntai et.al. \cite{SzaKov10} proved independently and
in different ways the following statement: If $P^k(\mathbf{X})$ is a $k$-width
junction tree pd approximation then there exists $P_{t-ch}^k(\mathbf{X})$
a $k$-th order $t$-cherry tree pd which gives at least as good approximation
as $P^k(\mathbf{X})$ does i.e.: 
\[KL\left( P\left( \mathbf{X}\right) ,P^k(%
\mathbf{X})\right) \geq KL\left( P\left( \mathbf{X}\right) ,P_{t-ch}^k(%
\mathbf{X})\right). 
\]
Hence this result we consider as search space the $k$-th
order $t$-cherry junction tree pd's.

In this part we first give a greedy algorithm to minimize the
Kullback-Leibler divergence between the true probability distribution and a
$t$-cherry junction tree pd given the $k$-th order marginal probability
distributions. We then compare our algorithm with Malvestuto's algorithm
from analytical point of view. Then we apply the two algorithms to the same
sample data proposed in \cite{Mal91}.

In \cite{SzaKov08} the authors give the following theorem.

\begin{theorem}
\label{theo:2}
The Kullback-Leibler divergence between the true $P(\mathbf{X})$ and the
approximation given by the $k$-width junction tree probability distribution $P(\mathbf{X}_{J})$, determined by the set
of clusters $\mathcal{C}$ and the set of separators $\mathcal{S}$ is :
\begin{equation}
\label{equ:equ1}
\begin{array}{rcl}
KL\left( P\left( \mathbf{X}\right) ,P_{J}\left( \mathbf{X}\right) \right)
& = & -H\left( \mathbf{X}\right) -\left( \sum\limits_{C\in \mathcal{C}}I\left(
\mathbf{X}_C\right) -\sum\limits_{S\in \mathcal{S}}\left( \nu _S-1\right) I\left(
\mathbf{X}_S\right) \right) \\
& + & \sum\limits_{i=1}^dH\left( X_i\right),
\end{array}
\end{equation}
where $I(\mathbf{X}_{C})=\sum\limits_{i\in \mathcal{C}} H\left(X_i\right) - H\left(
\mathbf{X}_{C}\right)$ represents the information content of the random vector $\mathbf{X}_{C}$
and similarly $I(\mathbf{X}_{S})=\sum\limits_{i\in \mathcal{S}} H\left(X_i\right) - H\left(
\mathbf{X}_{S}\right)$ represents the information content of the random vector $\mathbf{X}_{S}$.
\end{theorem}

In Formula (\ref{equ:equ1}) $-H\left( \mathbf{X}\right) +\sum\limits_{i=1}^dH\left(
X_i\right) =I\left( \mathbf{X}\right) $ is independent from the structure of
the junction tree. It is easy to see that minimizing the Kullback-Leibler
divergence means maximizing $\sum\limits_{C\in \mathcal{C}}I\left(
\mathbf{X}_C\right) -\sum\limits_{S\in \mathcal{S}}\left( \nu _S-1\right) I\left(
\mathbf{X}_S\right) $. We call this sum as \textit{weight of the junction tree pd}. As larger
this weight is, as better fits the approximation associated to the junction
tree pd to the true probability distribution. It is well known that $KL=0$ if $P(\mathbf{X})=P_{J}\left( \mathbf{X}\right)$.

In the case when the approximating probability distribution is given by a $k$-th order $t$-cherry
junction tree pd all of the clusters contain $k$ and 
all of the separators contain $k-1$ vertices in Formula (\ref{equ:equ1}).

Let $X=\left\{ X_1,\ldots ,X_d\right\} $a set of random variables.

\begin{definition}
\label{def:def4} 
We define the following concepts:

\begin{itemize}
\item  the search space:

\vspace{2mm}
$E=\left\{ \chi _{i_k\left( i_1,\ldots ,i_{k-1}\right) }=\left\{ \left\{
X_{i_k}\right\} ,\left\{ X_{i_1},\ldots ,X_{i_{k-1}}\right\} \right\}
|X_{i_1},\ldots ,X_{i_{k-1}},X_{i_k}\in X\right\} $,
\vspace{2mm}

\item  the independence set:

\vspace{2mm}
$\mathcal{F}=\phi \cup \left\{ t-\mbox{cherry junction tree structure}\right\} $,
\vspace{2mm}

\item  the weight function:

\vspace{2mm}
$w:E\rightarrow R\qquad w\left( \chi _{i_k\left( i_1,\ldots ,i_{k-1}\right)
}\right) =I\left( X_{i_1},\ldots ,X_{i_{k-1}},X_{i_k}\right) -I\left(
X_{i_1},\ldots ,X_{i_{k-1}}\right) $.
\end{itemize}
\end{definition}

\begin{algorithm}
\label{alg:2}
Sz\'{a}ntai-Kov\'{a}cs's greedy algorithm.

\textit{Input}: Elements of E and their weights which can be calculated
based on the $k$-th order marginal probability distributions.

\textit{Output}: set A which contains the clusters of the $k$-th order $t$%
-cherry juntion tree pd and the wheight of the $k$-th order $t$-cherry junction
tree pd.

\textit{The algorithm}:

$A:=\phi $

Sort $E$ into monotonically decreasing order by wheight $w$;

Choose $x={\arg \max }_{x\in E}\left( w\left( x\right) \right)$;

\qquad let $A:=A\cup \left\{ x\right\} ;\quad E:=E\backslash \left\{
x\right\} ;\quad w:=I\left( x\right)$;

Do for each $x\in E$ taken in monotonically decreasing order

\qquad if $A\cup \left\{ x\right\} \in \mathcal{F}$ then let $A:=A\cup \left\{ x\right\}
;\quad E:=E\backslash \left\{ x\right\} ;\quad w:=w+w\left( x\right) ;$

\qquad if the union of subsets of $A$ is $X$, then Stop;

\qquad else take the next element of $E$.
\end{algorithm}

In our $t$-cherry juntion tree terminology the KL divergence formula used by Malvestuto in his paper \cite{Mal91} is:

\begin{equation}
\label{equ:equ2}
KL\left( P\left( \mathbf{X}\right) ,P_{\mbox{t-ch}}\left( \mathbf{X}\right) \right)
=-H\left( \mathbf{X}\right) +\sum\limits_{C\in \mathcal{C}}H\left(
\mathbf{X}_C\right) -\sum\limits_{S\in \mathcal{S}}\left( \nu _S-1\right) H\left(
\mathbf{X}_S\right).
\end{equation}

In order to minimize the KL divergence Malvestuto had to minimize
\[
\sum\limits_{C\in \mathcal{C}}H\left( \mathbf{X}_C\right) -\sum\limits_{S\in 
\mathcal{S}}\left( \nu _S-1\right) H\left( \mathbf{X}_S\right)
\] 
in a greedy way.

Malvestuto's algorithm
uses the same search space $E$ and independence set $\mathcal{F}$. The
wheight function however is different:

$\omega :E\rightarrow R\qquad \omega \left( \chi _{i_k\left( i_1,\ldots
,i_{k-1}\right) }\right) =H\left( X_{i_1},\ldots ,X_{i_{k-1}},X_{i_k}\right)
-H\left( X_{i_1},\ldots ,X_{i_{k-1}}\right) $.

\begin{algorithm}
\label{alg:3}
\medskip Malvestuto's greedy algorithm.

\textit{Input}: Elements of E and their weights which can be calculated
based on the $k$-th order marginal probability distributions.

\textit{Output}: set A which contains the clusters of the $k$-th order $t$%
-cherry juntion tree probability distribution and the wheight of the $k$-th order $t$-cherry junction
tree.

$A:=\phi $

Sort $E$ into monotonically increasing order by wheight $w$;

Chose $x={\arg\min }_{x\in E}\left( H\left( x\right) \right) ;$

\qquad let $A:=A\cup \left\{ x\right\} ;\quad E:=E\backslash \left\{
x\right\} ;\quad \omega :=H\left( x\right) ;$

Do for each $x\in E$ taken in monotonically increasing order

\qquad if $A\cup \left\{ x\right\} \in \mathcal{F}$ then let $A:=A\cup \left\{ x\right\}
;\quad E:=E\backslash \left\{ x\right\} ;\quad \omega :=\omega +\omega
\left( x\right) ;$

\qquad if the union of subsets of $A$ is $X$, then Stop;

\qquad else take the next element of $E$.
\end{algorithm}

We present experimental results on the application of the two algorithms to
the probability distribution obtained from the sample data published in the paper \cite{Mal91}.  
These data contain informations on the structural habitat of grahami and opalinus
lizards. They were published originally by Bishop et al
\cite{BisFieHol75} and we give them in Table \ref{tab:1}.

\begin{table}[h]
\caption{Counts in structural habitat categories for Graham and Opalinus lizards}
\label{tab:1}       
\begin{tabular}{cccc}
\hline\noalign{\smallskip}
Cell ($X_1,X_2,X_3,X_4,X_5$) & Observed & Cell ($X_1,X_2,X_3,X_4,X_5$) & Observed \\
\noalign{\smallskip}\hline\noalign{\smallskip}
1 1 1 1 1  &  20  &  1 2 2 3 1  &  8  \\
2 1 1 1 1  &  13  &  2 2 2 3 1  &  4  \\
1 2 1 1 1  &  8   &  1 1 1 1 2  &  2  \\
2 2 1 1 1  &  6   &  1 2 1 1 2  &  3  \\
1 1 2 1 1  &  34  &  1 1 2 1 2  &  11 \\
2 1 2 1 1  &  31  &  2 1 2 1 2  &  5  \\
1 2 2 1 1  &  17  &  1 2 2 1 2  &  15 \\
2 2 2 1 1  &  12  &  2 2 2 1 2  &  1  \\
1 1 1 2 1  &  8   &  1 1 1 2 2  &  1  \\
2 1 1 2 1  &  8   &  1 2 1 2 2  &  1  \\
1 2 1 2 1  &  4   &  1 1 2 2 2  &  20 \\
1 1 2 2 1  &  69  &  2 1 2 2 2  &  4  \\
2 1 2 2 1  &  55  &  1 2 2 2 2  &  32 \\
1 2 2 2 1  &  60  &  2 2 2 2 2  &  5  \\
2 2 2 2 1  &  21  &  1 1 1 3 2  &  4  \\
1 1 1 3 1  &  4   &  1 2 1 3 2  &  3  \\
2 1 1 3 1  &  12  &  2 2 1 3 2  &  1  \\
1 2 1 3 1  &  5   &  1 1 2 3 2  &  10 \\
2 2 1 3 1  &  1   &  2 1 2 3 2  &  3  \\
1 1 2 3 1  &  18  &  1 2 2 3 2  &  8  \\
2 1 2 3 1  &  13  &  2 2 2 3 2  &  4  \\
\noalign{\smallskip}\hline
\end{tabular}
\end{table}

The data consists of observed counts for perch height ($<2'$ or %
$>2'$)--$X_1$, perch diameter ($<5''$ or $>5''$)--$X_2$, 
insolation (sun, shade)-- $X_3$, time of day categories (early,
midday, late) --$X_4$, lizard type (grahami, opalinus)--$X_5$ . 
The size of the contingeny table is $2\times2\times2\times3\times2$.

First we compare the goodness of fit of the 4-th order $t$-cherry junction tree 
found by Sz\'{a}ntai-Kov\'{a}cs's algorithm, then by Malvestuto's algorithm.

In Table \ref{tab:2} one can see the information contents of the marginal 
probability distribution of 4 random variables, 3 random variables and 
the weights used in Sz\'{a}ntai-Kov\'{a}cs's algorithm, ordered in decreasing way.

%
\begin{table}[h]
\caption{Illustration of Sz\'{a}ntai -- Kov\'{a}cs's algorithm}
\label{tab:2}       
\begin{tabular}{ccccc}
\hline\noalign{\smallskip}
Indices of the    & Indices of the      & $I(\mathbf{X}_{C})$ & $I(\mathbf{X}_{S})$ & $I(\mathbf{X}_{C})-I(\mathbf{X}_{S})$\\
cluster variables & separator variables &  &  & \\
\noalign{\smallskip}\hline\noalign{\smallskip}
{\bf 1 3 4 5}   &    1 3 5   &  {\bf 0.129381} & 0.045701 & 0.083680  \\
1 3 4 5   &    1 4 5   &   0.129381 & 0.047533 & 0.081848  \\
2 3 4 5   &    2 3 5   &   0.116608 & 0.035137 & 0.081470  \\
1 2 3 4   &    1 2 3   &   0.105531 & 0.026624 & 0.078907  \\
2 3 4 5   &    2 4 5   &   0.116608 & 0.038063 & 0.078544  \\
1 2 3 4   &    1 2 4   &   0.105531 & 0.029315 & 0.076216  \\
1 2 4 5   &    1 2 4   &   0.100251 & 0.029315 & 0.070936  \\
1 3 4 5   &    1 3 4   &   0.129381 & 0.066088 & 0.063294  \\
1 2 3 5   &    1 2 3   &   0.089070 & 0.026624 & 0.062446  \\
1 2 4 5   &    2 4 5   &   0.100251 & 0.038063 & 0.062187  \\
1 2 3 5   &    2 3 5   &   0.089070 & 0.035137 & 0.053933  \\
{\bf 1 2 4 5}   &   {\bf 1 4 5}   &   0.100251 & 0.047533 & {\bf 0.052718}  \\
\noalign{\smallskip}\hline
\end{tabular}
\end{table}

The junction tree obtained by Sz\'{a}ntai-Kov\'{a}cs's algorithm has two
clusters $\left\{ 1,3,4,5\right\} $, $\left\{ 1,2,4,5\right\} $ and one
separator $\left\{ 1,4,5\right\} $. The KL divergence in this
case is:
\[
\begin{array}{rcl}
KL&=&I\left( \mathbf{X}\right) -\left( I\left( \mathbf{X}_{\left\{
1,3,4,5\right\} }\right) -I\left( \mathbf{X}_{\left\{ 1,4,5\right\} }\right)
+I\left( \mathbf{X}_{\left\{ 1,2,4,5\right\} }\right) \right) \vspace{2mm}\\
&=&0.19519-(0.129381-0.047533+0.100251) = 0.013091.
\end{array}
\]

In Table \ref{tab:3} one can see the entropy of the marginal probability distribution of 4 random variables,
3 random variables and the weights used in Malvestuto's algorithm, ordered in increasing way.

The junction tree obtained by Malvestuto's algorithm has two clusters
$\left\{ 1,2,3,5\right\},$ $\left\{ 1,3,4,5\right\}$ and one separator 
$\left\{ 1,4,5\right\}$. 
The KL divergence in this case is:
\[
\begin{array}{rcl}
KL&=&-H\left( \mathbf{X}\right) +H\left( \mathbf{X}_{\left\{ 1,2,3,5\right\}
}\right) -H\left( \mathbf{X}_{\left\{ 1,3,5\right\} }\right) +H\left( 
\mathbf{X}_{\left\{ 1,3,4,5\right\} }\right)\vspace{2mm}\\
&=& -4.64164+3.288813-2.36849+3.743757 = 0.02244.
\end{array}
\]

%
\begin{table}[h]
\caption{Illustration of Malvestuto's algorithm}
\label{tab:3}       
\begin{tabular}{ccccc}
\hline\noalign{\smallskip}
Indices of the    & Indices of the      & $H(\mathbf{X}_{C})$ & $H(\mathbf{X}_{S})$ & $H(\mathbf{X}_{C})-H(\mathbf{X}_{S})$\\
cluster variables & separator variables &  &  & \\
\noalign{\smallskip}\hline\noalign{\smallskip}
{\bf 1 2 3 5}   &           &  {\bf 3.288813} &          &           \\
{\bf 1 3 4 5}   &  {\bf 1 3 5}    &  3.743757 & 2.368490 & {\bf 1.375267}  \\
2 3 4 5   &  2 3 5    &  3.783647 & 2.406170 & 1.377478  \\
1 2 3 4   &  1 2 3    &  3.943287 & 2.563246 & 1.380041  \\
1 2 4 5   &  1 2 5    &  4.046977 & 2.615873 & 1.431104  \\
\noalign{\smallskip}\hline
\end{tabular}
\end{table}

The two results of KL divergence reflect that the junction tree obtained by
our algorithm fits better to the probability distribution of the sample data.

If the task is fitting a third order $t$-cherry junction tree, then our
algorithm finds a $t$-cherry junction tree probability distribution, with $KL=0.0355415$. The
third order $t$-cherry junction tree given by Malvestuto's algorithm has the 
$KL=0.0375077$. 
The clusters found by our algorithm were
$\left\{ 3,4,5\right\} $, $\left\{ 1,4,5\right\} $, $\left\{ 1,2,5\right\} $
and those found by Malvestuto's algorithm were
$\left\{ 1,3,5\right\} $, $\left\{ 1,2,5\right\} $, $%
\left\{ 3,4,5\right\} $.

\section{Theorems related to the Sz\'{a}ntai-Kov\'{a}cs's algorithm}
\label{sec:4}

This part contains some theoretical discussions on the algorithm introduced,
regarding to assumptions related to the Markov network underlying the
variables.

As we remind in the preliminary part a triangulated graph can be represented
as a junction tree structure. If the graph is complete then the junction
tree has only one cluster.

If a graph is not triangulated, then by adding edges it can be transformed
into a triangulated graph. The problem of \textit{,,fill in as few edges as possible''} is
known to be NP complete (\cite{Yan81}). A greedy algorithm was given by Tarjan and Yanakakis
\cite{TarYan84}. 

If the vertices of a graph represent the indices of the
random variables of a Markov network with global Markov property then by adding new edges to the graph  
results a Markov network having the global Markov property, too.

If the graph associated to a Markov network is not complete then it can be
transformed into a triangulated graph by adding edges which is
equivalent with a junction tree structure, let say of order $k$. Since the
global Markov property holds for this graph the probability distribution can be
written as a product-division type, where the largest marginal probability distribution contains $k$
variables. A logical question which arises here is if the greedy algorithm
does find the $k$-th order junction tree which gives the true probability
distribution. For this question the answer is that under some assumption our
greedy algorithm guaranties the optimal solution, which in this context is
the true probability distribution.

We need the following assertion:

\begin{lemma}
\label{lem:1}
$H\left( X_1|X_2,\ldots ,X_k\right) =H\left( X_1\right) -\left[ I\left(
X_1,\ldots ,X_k\right) -I\left( X_2,\ldots ,X_k\right) \right] $.
\end{lemma}

\begin{proof}
\[
\begin{array}{rcl}
H\left( X_1|X_2,\ldots ,X_k\right) & = & H\left( X_1,X_2,\ldots ,X_k\right)
-H\left( X_2,\ldots ,X_k\right) \\
& = & H\left( X_1,X_2,\ldots,X_k\right) -\sum\limits_{i=1}^kH\left( X_i\right) \\
& - & \left( H\left( X_2,\ldots ,X_k\right) -\sum\limits_{i=2}^kH\left(X_i\right) \right) +H\left( X_1\right)\\
& = & H\left( X_1\right)-\left( I\left( X_1,X_2,\ldots ,X_k\right) -I\left( X_2,\ldots ,X_k\right)\right).
\end{array}
\]
\end{proof}

\begin{remark}
\label{rem:rem2}
It is easy to see that maximizing $I\left( X_1,\ldots ,X_k\right) -I\left(
X_2,\ldots ,X_k\right) $ is the same as maximizing $H\left( X_1\right)
-H\left( X_1|X_2,\ldots ,X_k\right) $.
\end{remark}

We introduce the following notations.

Let $\mathcal{K}=\left\{ K=\left\{ i_1,\ldots ,i_k\right\} |i_1,\ldots
,i_k\in V\right\} $ be the set of all possible $k$-element subsets of $V$. 

Let $M_K:\mathcal{K}\rightarrow R$ be defined as $M_K=\max_{i_s\in K}%
\left\{ I\left( \mathbf{X}_K\right) -I\left( \mathbf{X}_{K-\left\{
i_s\right\} }\right) \right\} $ and let 

\begin{equation}
\label{equ:equ3}
K^{*}=\arg \max_{K\in \mathcal{K}} M_K.
\end{equation}

We prove the following two theorems.

\begin{theorem}
\label{theo:3}
If \textbf{X }has a \textit{k}-th order \textit{t}-cherry tree
representation then $K^{*}$is a cluster of the junction tree.
\end{theorem}

\begin{proof}
We make the proof by contradiction. We suppose $K^{\star}=\left\{ i_1,\ldots
,i_k\right\} \notin \mathcal{C}$. Let us consider the smallest
subjunction tree which contains all the vertices $i_1,\ldots ,i_k$ at least
once. In this subjunction tree one of the vertices $i_1,\ldots ,i_k$ is a
simplicial vertex (a vertex which is contained in one cluster only). For
simplicity let this vertex be $i_1$ and the cluster which contains it $%
\left\{ i_1,s_1,\ldots ,s_{k-1}\right\} $, with $\left\{ s_1,\ldots
,s_{k-1}\right\} \neq \left\{ i_2,\ldots ,i_{k}\right\} $. We emphasize
here that it is not necessary that $\left\{ s_1,\ldots ,s_{k-1}\right\} \cap%
\left\{ i_2,\ldots ,i_{k}\right\} =\phi $.

Since $i_1$ is a simplicial vertex 
$X_{i_1}$ depends on all the other random variables of the subjunction tree only through
its neighbours $X_{s_1},\ldots ,X_{s_{k-1}}$, therefore

\[
H\left( X_{i_1}|X_{s_1},\ldots ,X_{s_{k-1}}\right) 
<H\left(X_{i_1}|X_{i_2},\ldots ,X_{i_{k}}\right).
\]
Using Lemma \ref{lem:1} this inequality is equivalent to:
\[
\begin{array}{rl}
&H\left( X_{i_1}\right) -\left[ I\left( X_{i_1},X_{s_1},\ldots
,X_{s_{k-1}}\right) -I\left( X_{s_1},\ldots ,X_{s_{k-1}}\right) \right] \vspace{2mm} \\
<&H\left( X_{i_1}\right) -\left[ I\left( X_{i_1},X_{i_2},\ldots
,X_{i_{k}}\right) -I\left( X_{i_2},\ldots ,X_{i_{k}}\right) \right]
\end{array}
\]
that is
\[
I\left( X_{i_1},X_{s_1},\ldots ,X_{s_{k-1}}\right) -I\left( X_{s_1},\ldots
,X_{s_{k-1}}\right) >I\left( X_{i_1},X_{i_2},\ldots ,X_{i_{k}}\right)
-I\left( X_{i_2},\ldots ,X_{i_{k}}\right)
\]
which is in contradiction with the hypothesis that $\left\{ i_1,\ldots
,i_k\right\} =K^{*}$.
\end{proof}

In the following we introduce the so called puzzle-algorithm, wich results a
special numbering of the verticies of t-cherry junction tree.

\begin{algorithm}
\label{alg:4}
Puzzle algorithm.

\textit{Input}: a $k$-th order $t$-cherry juncton tree $H\left( V,\Gamma
\right) $, (acyclic hypergraph with edges of size k, and separators of size
k-1)

\textit{Output}: a numbering $\left\{i_{1},\ldots,i_{d}\right\}$ of the verticies of $V=\left\{ 1,\ldots
,d\right\} $.

\textit{Step} 1. \textit{Initialization}.

\qquad Let $e_{i}\in\Gamma$, call it parent edge. The verticies belonging to the
parent edge are 

\qquad numbered in an arbitrary order by $i_{1},\ldots,i_{k}$.

\qquad $s:=k,\ \mathcal{S}_{s}:=\left\{ S_{i_{1}},\ldots,S_{i_{k}}\right\} $, where
for $j=1,\ldots,k,$ $S_{i_{j}}$ are all the $k-1$ 

\qquad element subset of $e_{i}$.

\textit{Step} 2. \textit{Iteration}.
 
\qquad Do $\Gamma=\Gamma\backslash e_{i}$.

\qquad Do if $\Gamma \neq \phi $ then take $e_i\in \Gamma $, which contains
one of the elements $S$ of $\mathcal{S}_s$.

\qquad \qquad Set $s:=s+1$,

\qquad \qquad assign $i_{s}$ to $i=e_{i}\backslash S$.

\qquad \qquad $\mathcal{S}_{s}=\mathcal{S}_{s-1}\cup\left\{
S_{i_{1}},\ldots,S_{i_{k}}\right\} $, where for $j=1,\ldots,k,$ $S_{i_{j}}$
are all the $k-1$ 

\qquad \qquad element subset of $e_i.$ Go to Step 2.

\qquad else Stop.
\end{algorithm}

\begin{definition}
\label{def:def5} 
The numbering $\left\{ i_{1},\ldots,i_{d}\right\} $
of the verticies of $V=\left\{ 1,\ldots,d\right\} ,$ obtained using Algorithm \ref{alg:4},
is called \textit{puzzle numbering}.
\end{definition}

\begin{theorem}
\label{theo:4} 
If the following two assumptions are fulfilled then the
Sz\'{a}ntai-Kov\'{a}cs algorithm finds the true probability distribution.

(i) The Markov network can be transformed into a $k$-th order $t$-cherry tree by adding some edges if it is necessary.

(ii) Starting from the parent cluster defined by (\ref{equ:equ3}) there exists a puzzle
numbering with the following property: for all $i_{r}<i_{s}$ and for any $S\in S_{r}$

\[
H\left( X_{i_{s}}\right)
-H\left( X_{i_{s}}|S\right) <H\left( X_{i_{r}}\right) -H\left(
X_{i_{r}}|S_{i_{r}}\right), 
\]
where $S_{i_{r}}$ is the
separator which separates $i_{r}$ from the tree containing the verticies $%
\left\{ i_{1},\ldots,i_{r-1}\right\} .$
\end{theorem}

\begin{proof}
We proved in Theorem \ref{theo:3} that the cluster $K^{\ast}$ which satisfies
(\ref{equ:equ3}) is a cluster of the junction tree associated to the Markov network. We
choose this cluster as parent edge.

Let us suppose that the Sz\'{a}ntai-Kov\'{a}cs Algorithm, has in the
constructed junction tree already $m-1$ verticies. We denote this set of verticies by $V_{m-1}$. 
The set of possible separators at this end is $\mathcal{S}_{m-1}$ .

The Sz\'{a}ntai-Kov\'{a}cs algorithm adds a new cluster by maximizing

\[
I\left( X_{i_{m}},\mathbf{X}_{S_{i}}\right) -I\left( \mathbf{X}%
_{S_{i}}\right) ,\mbox{ where }i_{m}\in V\backslash V_{m-1}\mbox{ and }%
S_{i}\in\mathcal{S}_{m-1} 
\]

According to Remark \ref{rem:rem1} this is equivalent with maximizing

\begin{equation}
\label{equ:equ4}
H\left( X_{i_{m}}\right) -H\left( X_{i_{m}}|S_{i}\right) \mbox{, where }%
i_{m}\in V\backslash V_{m-1}\mbox{ and }S_{i}\in\mathcal{S}_{m-1}. 
\end{equation}

We suppose now by contradiction that $i_m$is not connected to the existing
junction tree through $S_i$. Since the junction tree is a connected
hypergraph, there exist two possibilities:

\begin{itemize}
\item[1.] $i_m$ is separated from the existing tree $T_{m-1}$ by another separator $%
S_j\in \mathcal{S}_{m-1};$

\item[2.] There exists $i_n\in V\backslash V_{m-1}$ which is connected with the
existing junction tree by $S_i\in \mathcal{S}_{m-1}$, and the cluster $%
\left( S_i,i_n\right) $ is on the path between the existing tree $T_{m-1}$ and the
cluster which contains $i_m$.
\end{itemize}

Now we pove that none of the two possibilities can occur.

\begin{itemize}
\item[1.] If $i_m$ is separated from the existing tree $T_{m-1}$ by another
separator $S_j\in \mathcal{S}_{m-1}$ then according to the global Markov property we have:

\[
P\left( \mathbf{X}_{T_{m-1}},X_{i_m}\right) =\frac{P\left( \mathbf{X}%
_{T_{m-1}}\right) P\left( \mathbf{X}_{S_j}X_{i_m}\right) }{P\left( \mathbf{%
X}_{S_j}\right) }. 
\]

This implies that the Kullback Leibler between 
\[P\left( \mathbf{X}%
_{T_{m-1}},X_{i_m}\right) \mbox{ and } \frac{P\left( \mathbf{X}_{T_{m-1}}\right) 
P\left( \mathbf{X}_{S_j}X_{i_m}\right) }{P\left( \mathbf{X}_{S_j}\right) }
\]
is 0:
\[
KL=I\left( \mathbf{X}_{T_{m-1}},X_{i_m}\right) -\left( I\left( \mathbf{X}%
_{T_{m-1}}\right) +I\left( \mathbf{X}_{S_j}X_{i_m}\right) -I\left( \mathbf{X}%
_{S_j}\right) \right) =0.
\]
Thus
\begin{equation}
\label{equ:equ5}
I\left( \mathbf{X}_{S_j}X_{i_m}\right) -I\left( \mathbf{X}%
_{S_j}\right) =I\left( \mathbf{X}_{T_{m-1}},X_{i_m}\right) -I\left( \mathbf{X}%
_{T_{m-1}}\right).
\end{equation}
On the other hand if $S_i$ does not separate $i_m$ from the existing tree
then the KL between
\[
P\left( \mathbf{X}_{T_{m-1}},X_{i_m}\right) \mbox{ and } \frac{P\left( \mathbf{X}%
_{T_{m-1}}\right) P\left( \mathbf{X}_{S_i}X_{i_m}\right) }{P\left( \mathbf{%
X}_{S_i}\right) }
\] 
is positive:
\[
KL=I\left( \mathbf{X}_{T_{m-1}},X_{i_m}\right) -\left( I\left( \mathbf{X}%
_{T_{m-1}}\right) +I\left( \mathbf{X}_{S_i},X_{i_m}\right) -I\left( \mathbf{X}%
_{S_i}\right) \right) >0.
\]
Thus
\begin{equation}
\label{equ:equ6}
I\left( \mathbf{X}_{S_i}X_{i_m}\right) -I\left( \mathbf{X}%
_{S_i}\right) <I\left( \mathbf{X}_{T_{m-1}},X_{i_m}\right) -I\left( \mathbf{X}%
_{T_{m-1}}\right).
\end{equation}

From (\ref{equ:equ5}) and (\ref{equ:equ6}) we have $I\left( \mathbf{X}%
_{S_i}X_{i_m}\right) -I\left( \mathbf{X}_{S_i}\right) <I\left( \mathbf{X}%
_{S_j}X_{i_m}\right) -I\left( \mathbf{X}_{S_j}\right) $.

According to Remark \ref{rem:rem1} this implies 
\[
H\left( X_{i_m}\right) -H\left( X_{i_m}|%
\mathbf{X}_{S_i}\right) <H\left( X_{i_m}\right) -H\left( X_{i_m}|\mathbf{X}%
_{S_j}\right) 
\]
which is in contradiction with maximization of (\ref{equ:equ4}).

\item[2.] If on the path between the existing $T_{m-1}$ tree and the cluster which
contains $i_m$there exists a cluster $\left( S_i,i_n\right) $ , where $%
i_n\in V\backslash V_{m-1}$,and$S_i\in \mathcal{S}_{m-1}$, then according to
the puzzle numbering $i_n<i_m$ . Using and (ii) we have:

\[
H\left( X_{i_m}\right) -H\left( X_{i_m}|S\right) <H\left( X_{i_n}\right)
-H\left( X_{i_n}|S_{i_n}\right) 
\]

for any $S\in \mathcal{S}_{m-1}$, and $S_{i_n}$ $\in \mathcal{S}_{m-1}$
separator between $i_{n}$and the existing tree $T_{m-1}$ $.$ This is
in contradiction with maximization of (\ref{equ:equ4}).
\end{itemize}
\end{proof}

\begin{theorem}
\label{theo:5} If the the best aproximating $k$-th order t-cherry
probability distribution has a puzzle numbering which starting from the
parent cluster defined by (\ref{equ:equ3}) satisfies (i) and (ii) then the Sz\'{a}ntai-Kov%
\'{a}cs Algorithm finds the best aproximating $k$-th order $t$-cherry
probability distribution.

\begin{itemize}
\item[i)] for all $i_r<i_s$, for any $S\in \mathcal{S}_r$, $H\left( X_{i_s}\right) -H\left(
X_{i_s}|S\right) <H\left( X_{i_r}\right) -H\left( X_{i_r}|S_{i_r}\right) $,
where $S_{i_r}\in \mathcal{S}_r$ is the separator which separates $i_r$ from the tree
containing the verticies $\left\{ i_1,\ldots ,i_{r-1}\right\} $

\item[ii)] for all $i_r>k$
\[
X_{i_r}=\arg \min_{i\in V\backslash \left\{ i_1,\ldots ,i_{r-1}\right\}
}KL\left( P_{app}\left( X_{i_1},\ldots ,X_{i_{r-1}},X_i\right) ,P\left(
X_{i_1},\ldots ,X_{i_{r-1}},X_i\right) \right) 
\]
\end{itemize}
\end{theorem}

\begin{proof}
Let the cluster $K^{*}$ which satisfies (\ref{equ:equ3}) the first cluster of the
junction tree. We choose this cluster as parent edge.

Let us suppose that the Sz\'{a}ntai-Kov\'{a}cs Algorithm, has in the
constructed junction tree already $m-1$ verticies. The set of possible
separators at this end is $\mathcal{S}_{m-1}$ .

The Sz\'{a}ntai-Kov\'{a}cs algorithm adds a new cluster by maximizing

\begin{equation}
\label{equ:equ7}
I\left( X_{i_m},\mathbf{X}_{S_i}\right) -I\left( \mathbf{X}_{S_i}\right) ,%
\mbox{ where }i_m\in V\backslash V_{m-1}\mbox{ and }S_i\in \mathcal{S}%
_{m-1} 
\end{equation}

According to Remark \ref{rem:rem1} this is equivalent with maximizing

\begin{equation}
\label{equ:equ8}
H\left( X_{i_m}\right) -H\left( X_{i_m}|S_i\right) \mbox{, where }i_m\in
V\backslash V_{m-1}\mbox{ and }S_i\in \mathcal{S}_{m-1}
\end{equation}

We suppose now by contradiction that in the best approximating junction tree $i_m$ is not
connected to the existing junction tree through $S_i$. Since the best approximating junction
tree is a connected hypergraph there exist two possibilities:

\begin{itemize}
\item[1.] $i_m$ is separated from the existing tree $T_{m-1}$ by another separator $%
S_j\in \mathcal{S}_{m-1};$

\item[2.] In the best approximating junction tree there exists $i_n\in V\backslash V_{m-1}$ which is connected with the
existing junction tree by $S_i\in \mathcal{S}_{m-1}$, and the cluster $%
\left( S_i,i_n\right) $ is on the path between the existing tree and the
cluster which contains $i_m$.
\end{itemize}

Now we pove that none of the two possibilities can occur.

\begin{itemize}
\item[1.] If $i_m$ is separated from the existing tree $T_{m-1}$ by another
separator $S_j$ then according to the Markov property we have:

\[
P_{app}^j\left( \mathbf{X}_{T_{m-1}},X_{i_m}\right) =\frac{P\left( \mathbf{X}%
_{T_{m-1}}\right) P\left( \mathbf{X}_{S_j}X_{i_m}\right) }{P\left( \mathbf{%
X}_{S_j}\right) }. 
\]

This implies that the Kullback Leibler between 
\[
P\left( \mathbf{X}%
_{T_{m-1}},X_{i_m}\right)  \mbox{ and } P_{app}^j\left( \mathbf{X}_{T_{m-1}},X_{i_m}\right)
\]
is given by:

\[
\begin{array}{l}
KL\left( P_{app}^j\left( \mathbf{X}_{T_{m-1}},X_{i_m}\right),P\left( \mathbf{X}_{T_{m-1}},X_{i_m}\right)\right)\vspace{3mm} \\
=I\left( \mathbf{X}_{T_{m-1}},X_{i_m}\right)
-\left( I\left( \mathbf{X}_{T_{m-1}}\right) +I\left( \mathbf{X}%
_{S_j}X_{i_m}\right) -I\left( \mathbf{X}_{S_j}\right) \right)
\end{array}
\]

According to (\ref{equ:equ7}) 
\[I\left( \mathbf{X}_{S_j}X_{i_m}\right) -I\left( 
\mathbf{X}_{S_j}\right) <I\left( \mathbf{X}_{S_i}X_{i_m}\right) -I\left( 
\mathbf{X}_{S_i}\right)
\]
and this implies that
\[
\begin{array}{l}
KL\left( P_{app}^i\left( \mathbf{X}_{T_{m-1}},X_{i_m}\right),P\left( \mathbf{X}_{T_{m-1}},X_{i_m}\right)\right)\vspace{3mm} \\
=I\left( \mathbf{X}_{T_{m-1}},X_{i_m}\right)-\left( I\left( \mathbf{X}_{T_{m-1}}\right) +I\left( \mathbf{X}%
_{S_i}X_{i_m}\right) -I\left( \mathbf{X}_{S_i}\right) \right) <KL\left(
P_{app}^j,P\right).
\end{array}
\] 
This is in contradiction with (ii).

\item[2.] If on the path between the existing $T_{m-1}$ tree and the cluster which
contains $i_m$there exists a cluster $\left( S_i,i_n\right) $ , where $%
i_n\in V\backslash V_{m-1}$,and$S_i\in \mathcal{S}_{m-1}$, then according to
the puzzle numbering $i_n<i_m$ and (i) we have:
\[
H\left( X_{i_m}\right) -H\left( X_{i_m}|S\right) <H\left( X_{i_n}\right)
-H\left( X_{i_n}|S_{i_n}\right) 
\]
for any $S\in \mathcal{S}_{m-1}$, and $S_{i_n}$ $\in \mathcal{S}_{m-1}$
separator between $i_{n}$and the existing tree $T_{m-1}$ $.$ This is
in contradiction with maximizing (\ref{equ:equ8}).
\end{itemize}
\end{proof}

\section{Conclusions}
\label{sec:5}

We give in this paper a greedy algorithm for fitting $k$%
-width junction tree approximation by minimizing the Kullback-Leibler
divergence. The problem of finding the best approximation of this kind is
generally an NP-hard problem. We reduce the search space to the so called $k$-th
order $t$-cherry junction tree probability distributions. We then compare
our algorithm to Malvestuto's algorithm. We proved that our algorithm in the first step finds
a cluster which belongs to the junction tree. Malvestuto's
algorithm has not guarantee for this. Beside this our formula for
Kullback-Leibler divergence (\ref{equ:equ1}) detached a greater part which does not depend on
the structure of the tree than Malvestuto's formula (\ref{equ:equ2}).

We proved that under some assumptions our algorithm finds the optimal
solution.

By discovering the $t$-cherry junction tree probability distribution
assigned to a Markov network we can obtain many information on the
dependence structure underlying the random variables. This information can
be used for storing the data in lower dimensional contingency tables. The
method can be applied in classification problems where it is possible to
select the ''informative'' variables which influence directly the
classification variable, see \cite{SzaKov10}.




\end{document}